\documentclass[12pt]{amsart}
\usepackage{amssymb, latexsym}
\usepackage{hyperref}
\usepackage{graphics}
\usepackage{epsfig}
\usepackage{pstricks,pst-node,pst-text,pst-3d}

\newcommand {\R}{\mathbb {R}}
\newcommand {\Z}{\mathbb {Z}}

\newcommand {\cH}{\mathcal {H}}
\newcommand {\cK}{\mathcal {K}}
\newcommand {\cG}{\mathcal {G}}
\newcommand {\cT}{\mathcal {T}}
\newcommand{\Go}{\Gamma_{\!0}}

\DeclareMathOperator{\length}{length}
\DeclareMathOperator{\Sk}{Sk}

\newtheorem{theorem}{Theorem}[section]
\newtheorem{lemma}[theorem]{Lemma}
\newtheorem{proposition}[theorem]{Proposition}

\theoremstyle{definition}
\newtheorem{remark}[theorem]{Remark}
\newtheorem{definition}[theorem]{Definition}
\numberwithin{equation}{section}

\begin{document}

\title[A doubling measure on  $\R^d$ can charge a rectifiable curve]{A doubling
 measure on $\R^d$ can charge\\a rectifiable curve}

\author{John Garnett}\address{John Garnett, Department of Mathematics, University of California, Los Angeles}
\author{Rowan Killip}\address{Rowan Killip, Department of Mathematics, University of California, Los Angeles}
\author{Raanan Schul}\address{Raanan Schul, Department of Mathematics, State University of New York, Stony Brook}

\begin{abstract}
For $d\geq 2$, we construct a doubling measure $\nu$ on $\R^d$ and a rectifiable curve $\Gamma$
such that $\nu(\Gamma)>0$.
\end{abstract}

\maketitle

%%%%%%%%%%%%%%%%%%%%%%%%%%%%%%%%%%%%%%%%%%%%%%%%%%%%%%%%
\section{Introduction}

A Borel measure $\nu$ on $\R^d$ is said to be \emph{doubling} if there is a constant $C_\nu<\infty$ such that
for any $x\in \R^d$ and $0<r<\infty$ we have
\begin{equation}
\nu(B(x,2r))\leq
C_\nu(B(x,r))\,
\end{equation}
where $B(x,r)$ is the ball $\{y: |y-x| < r\}.$ 
A \emph{rectifiable curve} is a continuous map $\gamma:[0,1]\to\R^d$ with
$$
    \length(\gamma) := \sup_{0\leq t_0\leq\cdots\leq t_n\leq 1} \sum_{j=1}^n |\gamma(t_j)-\gamma(t_{j-1})| < \infty.
$$
By reparametrization, one may assume that $\gamma$ is Lipschitz with constant equal to $\length(\gamma)$. We will also make use
of the following simple (and well-known) criterion: a compact set $\Gamma$ is the image of a rectifiable curve if and
only if it is connected and $\cH^1(\Gamma)<\infty$.  Indeed, one may choose $\gamma$ so that $\length(\gamma) \leq C
\cH^1(\Gamma)$; see, for example, \cite{DS,Falconer}.  Here and below, $\cH^1$ denotes the one-dimensional Hausdorff
measure.

The purpose of this note is to prove

\begin{theorem}\label{T:1}
Let $d\geq 2$.
There exists a doubling measure $\nu$ on $\R^d$ and a rectifiable curve $\Gamma$ such that $\nu(\Gamma)>0$.
\end{theorem}

%The question of whether a doubling measure can charge a rectifiable curve was posed to the third author by Mario Bonk.
We note that doubling measures cannot charge even slightly more regular curves; indeed the authors' initial belief was
that a rectifiable curve could not carry any weight.
As discussed in \cite[\S I.8.6]{Stein} doubling measures give zero weight to any smooth hyper-surface.  The argument,
based on Lebesgue's density theorem (for $\nu$), adapts without difficulty to show that for any connected set~$\Gamma$,
$$
\nu\bigl( \{x\in\Gamma:\ \liminf_{r\to 0} r^{-1} \cH^1 (B(x,r)\cap \Gamma) <\infty \} \bigr) = 0.
$$
Therefore if $\Gamma$ is a rectifiable curve, then $\nu|_\Gamma$ must be singular to $\cH^1|_\Gamma$.
Similarly, no doubling measure can charge an Ahlfors regular curve.

We will prove Theorem~\ref{T:1} by explicitly constructing a measure and a rectifiable curve.  
%Our measure $\nu$ will be the $d$-fold product of a doubling measure $\mu$ on $\R$.

\subsection*{Acknowledgements:}  
The question of whether a doubling measure can charge a rectifiable curve was posed to the third author by Mario Bonk. It seems to have been communicated to several people also by Saara Lehto and  Kevin Wildrick, and may have originated with the late Juha Heinonen.  
We are grateful to Jonas Azzam for a valuable critique of our
original approach to this problem.

While working on this paper we were supported in part by various NSF grants:
John Garnett, by  DMS-0758619 and DMS-0714945, 
Rowan Killip, by DMS-0401277 and DMS-0701085,
and 
Raanan Schul, by  DMS-0800837 (renamed to DMS-0965766) and  DMS-0502747.
Part of the work on this paper was done while the first author was a guest at the Centre de Recerca Matem\`atica in Barcelona.

%%%%%%%%%%%%%%%%%%%%%%%%%%%%%%%%%%%%%%%%%%%%%%%%%%%%%%%%
\section{Proof}
%%%%%%%%%%%%%%%%%%%%%%%%%%%%%%%%%%%%%%%%%%%%%%%%%%%%%%%%
\subsection{The Measure}
%As noted in the introduction, 
Our measure $\nu$ will be the $d$-fold product of a doubling measure $\mu$ on $\R$.  
The latter is constructed
by a simple iterative procedure that we will now describe.  It may be viewed as a variant of the classic Riesz product
construction and a `lift the middle' idea of Kahane (cf. \cite{Kahane}).  A very general form of this construction
appears in \cite{Wu}.

Let $h:\R\to\R$ be the $1$-periodic function
\begin{equation*}
  h(x) = \begin{cases} \,\,2 &: x\in [\tfrac13,\tfrac23)+\Z \\ - 1 & :\text{otherwise.}  \end{cases}
\end{equation*}
Then given $\delta\in(0,\tfrac13]$, we define $\mu$ as the weak-$*$ limit of
$$
d\mu_n := \prod_{j=0}^{n-1} \bigl[1+(1-3\delta)h(3^jx)\bigr]\; dx
$$
When $\delta=1/3$, $\mu$ is Lebesgue measure.

By viewing points $x\in\R$ in terms of their ternary (i.e., base 3) expansion, we may interpret $\mu$ as the result
of a sequence of independent trials.  More precisely, let $\cT_n$ denote the collection of triadic intervals of
size $3^{-n}$, that is,
\begin{equation}
\cT_n=\bigl\{\bigl[i 3^{-n}, (i+1)3^{-n}\bigr) :  i\in \Z \bigr\}.
\end{equation}
Then the measure of a triadic interval $I=[i 3^{-n}, (i+1)3^{-n})$ is related to that of its
parent $\hat I$, the unique interval in $\cT_{n-1}$ containing $I$, by
\begin{equation}
\mu(I)=
\begin{cases}
  (1-2\delta)\mu(\hat{I})   & : i\equiv 1 \mod 3 \\
  \delta\mu(\hat{I})        & : \text{otherwise.}
\end{cases}
\end{equation}
Coupled with the fact that $\mu([i,i+1))=1$ for $i\in\Z$, condition (2.2)  uniquely determines $\mu$.  In particular, we
note that if $j$, $n\geq 0$, and $0\leq i < 3^n$ are integers, then
\begin{equation}\label{e__triadic}
\mu\bigl( [j+i3^{-n},j+(i+1)3^{-n}) \bigr) = \delta^{n-k(i)} (1-2\delta)^{k(i)}
\end{equation}
where $k(i)$ is the number of times the digit $1$ appears in the ternary expansion of $i$.

We claim that $\mu$ is a doubling measure on $\R$. First let $I$ and $J$ be adjacent triadic intervals of equal size.
By \eqref{e__triadic} we have that $\mu(I)/\mu(J) \leq \frac{1-2\delta}{\delta}$. Several applications of this
shows that $\mu(I)/\mu(J) \leq C(\delta)$ for any pair $I$ and $J$ of adjacent intervals of equal size. Thus $\mu$ is
doubling.

Let $\nu$ be the product measure $\mu\times\cdots\times\mu$ on $\R^d$.  This is a doubling measure:
$\nu(I_1\times...\times I_d) \leq C(\delta)^d \, \nu(J_1\times...\times J_d)$ for any $d$ pairs of identical or adjacent
intervals $I_l,J_l$ that obey  $|I_l|=|J_l|$. Indeed, this holds even without the requirement
that $|I_l|=|I_{l'}|$ for $l\neq l'$.

%%%%%%%%%%%%%%%%%%%%%%%%%%%%%%%%%%%%%%%%%%%%%%%%%%%%%%%%
\subsection{The Basic Building Blocks}
%%%%%%%%%%%%%%%%%%%%%%%%%%%%%%%%%%%%%%%%%%%%%%%%%%%%%%%%

\begin{definition}
Given integer parameters $0\leq k \leq n$, we define $K(n,k)\subset [0,1)$ via
\begin{equation}\label{E:K defn}
K(n,k) = \cup \bigl\{ I \in \cT_n : I\subseteq[0,1) \text{ and }
\mu(I) \geq \delta^k (1-2\delta)^{n-k} \bigr\}.
\end{equation}
Equivalently, 
if $\delta<\frac13$, 
$K(n,k)$ is the set of those $x\in[0,1)$ whose ternary
expansion contains at most $k$ zeros or twos amongst
the first $n$ digits.
\end{definition}

\begin{lemma}
For $2\delta n \leq k \leq \tfrac23 n$ and $K=K(n,k)$ defined as in
\eqref{E:K defn}, we have
\begin{equation}\label{E:mu of K}
1 - \mu(K) \leq \exp\bigl\{ -2n\bigl(\tfrac{k}n -2\delta)^2 \bigr\}
\end{equation}
and
\begin{equation}\label{E:length of K}
|K| \leq 3^{-n} e^{k[ 1+ \log(\delta^{-1})]}
\end{equation}
\end{lemma}

\begin{proof}
Both inequalities rest on standard estimates for tail probabilities
for the binomial distribution.  
%These are proved by the usual large deviation technique:
These are proved by the usual large deviation technique of Cram\'er
(cf. \cite[Theorem 1.3.13]{Stroock}):  
\begin{align*}
\sum_{m\geq an}^n \tbinom{n}{m} p^m (1-p)^{n-m}
&\leq \inf_{t\geq 0} \sum_{m=0}^n {\textstyle \binom{n}{m}} p^m (1-p)^{n-m} e^{(m-an)t} \\
    &= \inf_{t\geq 0}\,  \bigl[ e^{-at}( 1 - p + p\; e^{t} )\bigr]^n
\end{align*}
This infimum can be determined exactly and for $0<p\leq a < 1$ we obtain 
$$
\sum_{m\geq an}^n \tbinom{n}{m} p^m (1-p)^{n-m} \leq e^{-n H(a,p)}
$$
where
$$
 H(a,p) = a\log\bigl(\tfrac ap\bigr) + (1-a)\log\bigl(\tfrac{1-a}{1-p}\bigr).
$$

For \eqref{E:mu of K} we set $a=k/n$ and $p=2\delta$ and make use of the fact that
$$
H(a,p) \geq 2(a-p)^2.
$$
Indeed,  $H$ and $\partial_a H$ vanish at $a=p$, while $\partial_a^2 H = a^{-1}(1-a)^{-1} \geq 4$.

To obtain \eqref{E:length of K}, we set $p=\frac13$ and
$a=\frac{n-k}{n}$.  We simplified the answer by using
$$
H(a,p) \geq \log\bigl(\tfrac1p\bigr) - (1-a)\Bigl[
\log\bigl(\tfrac{1-p}{p}\bigr) + 1 +
\log\bigl(\tfrac{1}{1-a}\bigr)\Bigr],
$$
which amounts simply to $a\log(a)+1-a = - \int_a^1 \log(t)\,dt \geq 0$.
\end{proof}

\begin{remark}
Choosing $\delta < 2/9$ and $k=3\delta n$ and sending $n\to\infty$, we see by Lemma 2.2  that $\mu$ gives all its weight to
a set of Hausdorff dimension $O(\delta\log(\delta^{-1}))$.  The precise dimension of $\mu$ is not important
to us; however, we will exploit the fact that it can be made as small as we wish by sending $\delta\downarrow0$.
Indeed, the product measure $\nu$ cannot charge a set of Hausdorff dimension one (not to mention
a rectifiable curve) unless $\mu$ gives positive weight to a set of dimension $d^{-1}$ or smaller.
\end{remark}

By definition, $K(n,k)$ is a union of intervals from $\cT_n$.  Correspondingly, the $d$-fold Cartesian product $K(n,k)^d$
can be viewed as a union of triadic cubes $Q\subseteq \R^d$ (with side-length $3^{-n}$).  We denote this collection of
cubes by $\cK^d(n,k)$.  By~\eqref{E:length of K},
\begin{equation}\label{E:size of cK}
  \# \cK^d(n,k) \leq e^{kd[ 1+ \log(\delta^{-1})]}
\end{equation}
Similarly, we write $\cG(n,k)$ for the gaps in $K(n,k)$, that is, the bounded connected components of $\R\setminus K(n,k)$.
As each gap has a right end-point, \eqref{E:length of K} gives
\begin{equation}\label{E:number of gaps}
  \# \cG(n,k) \leq e^{k[ 1+ \log(\delta^{-1})]}.
\end{equation}
Note also that $|\cup \cG(n,k)| \leq 1$, as $K(n,k)\subseteq [0,1)$.

We now define a curve $\Gamma(n,k)\subset \R^d$ which visits each
cube $Q\in \cK^d(n,k)$.  Actually, we merely construct a connected family of line segments $\Gamma(n,k)$
that do this, and bound its total length.  As noted in the introduction, all segments in
$\Gamma(n,k)$ can be traversed by a single curve of comparable total length.

The family $\Gamma(n,k)$ is the union of skeletons of rectangular boxes, where we define the \emph{skeleton} of a box is
$$
\Sk(I_1\times\cdots\times I_d) = \bigcup_{j=1}^d \ \partial I_1 \times \cdots\times \partial I_{j-1} \times I_j
    \times \partial I_{j+1} \times \cdots \times \partial I_d.
$$
Thus $\Sk(Q)$ is the union  of the edges --- as opposed to vertices, faces, 3-faces, etc. --- of the box $Q$.  With this
notation,
$$
\Gamma(n,k) =  \bigcup_{Q\in\cK^d(n,k)} \Sk(Q)  \quad\cup
    \bigcup_{I_1,\ldots,I_d \in \cG(n,k)} \Sk(I_1\times\cdots\times I_d).
$$
Note that $\Gamma(n,k)$ is connected.
We now estimate the total length of this set.

\begin{lemma}[The length of the $\Gamma(n,k)$]
Assuming $2\delta n \leq k \leq \tfrac23 n$,
\begin{equation}\label{E:length Gamma n k}
\cH^1(\Gamma(n,k))  \leq d 2^{d} e^{d k [ 1 + \log(\delta^{-1})]}.
\end{equation}
\end{lemma}
\begin{proof}
By \eqref{E:size of cK} and \eqref{E:number of gaps},
\begin{align*}
\cH^1(\Gamma(n,k)) &\leq \sum_{Q\in\cK^d(n,k)} \cH^1 \bigl(\Sk(Q)\bigr)
    + \!\!\!\!\sum_{I_1,\ldots,I_d \in \cG(n,k)}\!\!\!\!\cH^1\bigl(\Sk(I_1\times\cdots\times I_d)\bigr) \\[2mm]
&=  d2^{d-1} 3^{-n} \bigl[\#\,\cK^d(n,k)\bigr] + \,d 2^{d-1} \bigl[\#\,\cG(n,k)\bigr]^{d-1} 
    \!\!\!\sum_{I\in\cG(n,k)}\!\!\! |I| \\
&\leq d 2^{d-1} 3^{-n} e^{kd[ 1+ \log(\delta^{-1})]} + d 2^{d-1} e^{k(d-1)[ 1+ \log(\delta^{-1})]}
\end{align*}
which easily yields \eqref{E:length Gamma n k}.
\end{proof}

\subsection{The Curve}
Using $\Gamma(n,k)$ as a building-block, we now explain the iterative construction of the full curve~$\Gamma$.
It depends upon a collection of parameters $\{n_j,k_j\}_{j=1}^\infty$.
The guiding principle is to replace each cube in $\cK^d(n_j,k_j)$ by rescaled/translated copies of
$\cK^d(n_{j+1},k_{j+1})$ and $\Gamma(n_{j+1},k_{j+1})$.
See Figure \ref{f__K_figure}.

To this end, we define a version $\Gamma_Q(n,k)$ of $\Gamma(n,k)$ adapted to any cube~$Q$:
$$
 \Gamma_Q(n,k) = A_Q\bigl(  \Gamma(n,k) \bigr)
$$
where $A_Q$ is the affine transformation that maps $[0,1)^d$ to $Q$.  Similarly, we inductively define
$$
\cK_0 =\bigl\{[0,1)^d\bigr\} \ \ \text{and}\ \ \cK_{l}
    = \bigcup_{Q\in\cK_{l-1}} \bigl\{ A_Q(Q') : Q'\in \cK^d(n_{l},k_{l}) \bigr\}  \ \ \text{for $l\geq 1$.}
$$
Thus $\cK_{l}$ is the collection of cubes remaining after the $l^{\rm \,th}$ iteration in the construction of~$\Gamma$.  Subsequent iterations
will not modify $\Gamma$ outside their union, $$K_{l} = \cup\{Q:Q\in\cK_{l}\}.$$  We note that the cubes in $\cK_{l}$ have
disjoint interiors, and that by \eqref{E:size of cK},
\begin{equation}\label{E:size K_l}
\begin{aligned}
\# \cK_{l}  &\leq [\# \cK_{l-1}] \exp\bigl\{k_l d[ 1+ \log(\delta^{-1})]\bigr\} \\
    &\leq \exp\bigl\{ (k_1+\cdots+k_l) d[ 1+ \log(\delta^{-1})] \bigr\}
\end{aligned}
\end{equation}

\begin{figure}[t]
\scalebox{0.3}{\includegraphics*{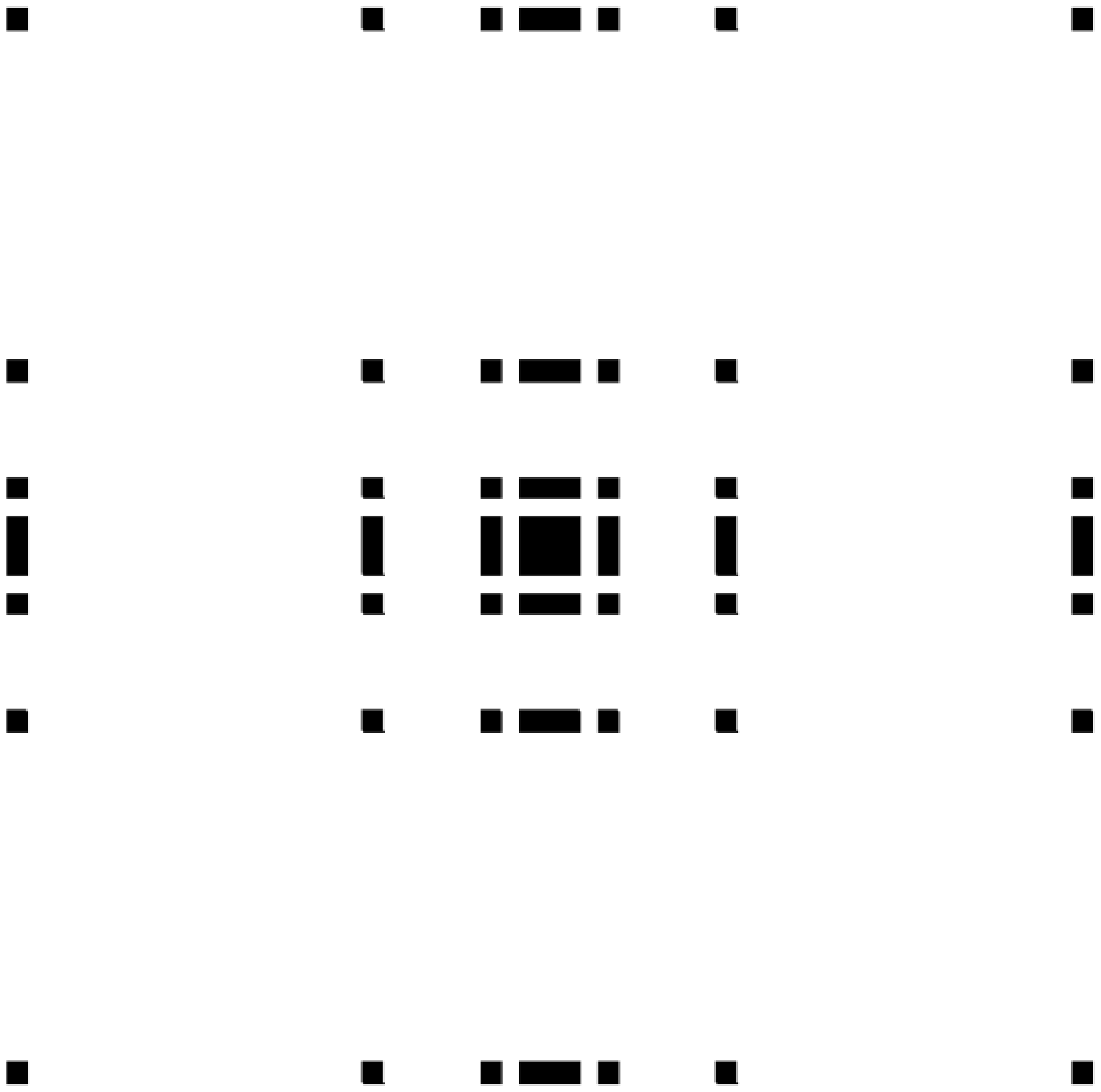}}
\hspace{3cm}
\scalebox{0.3}{\includegraphics*{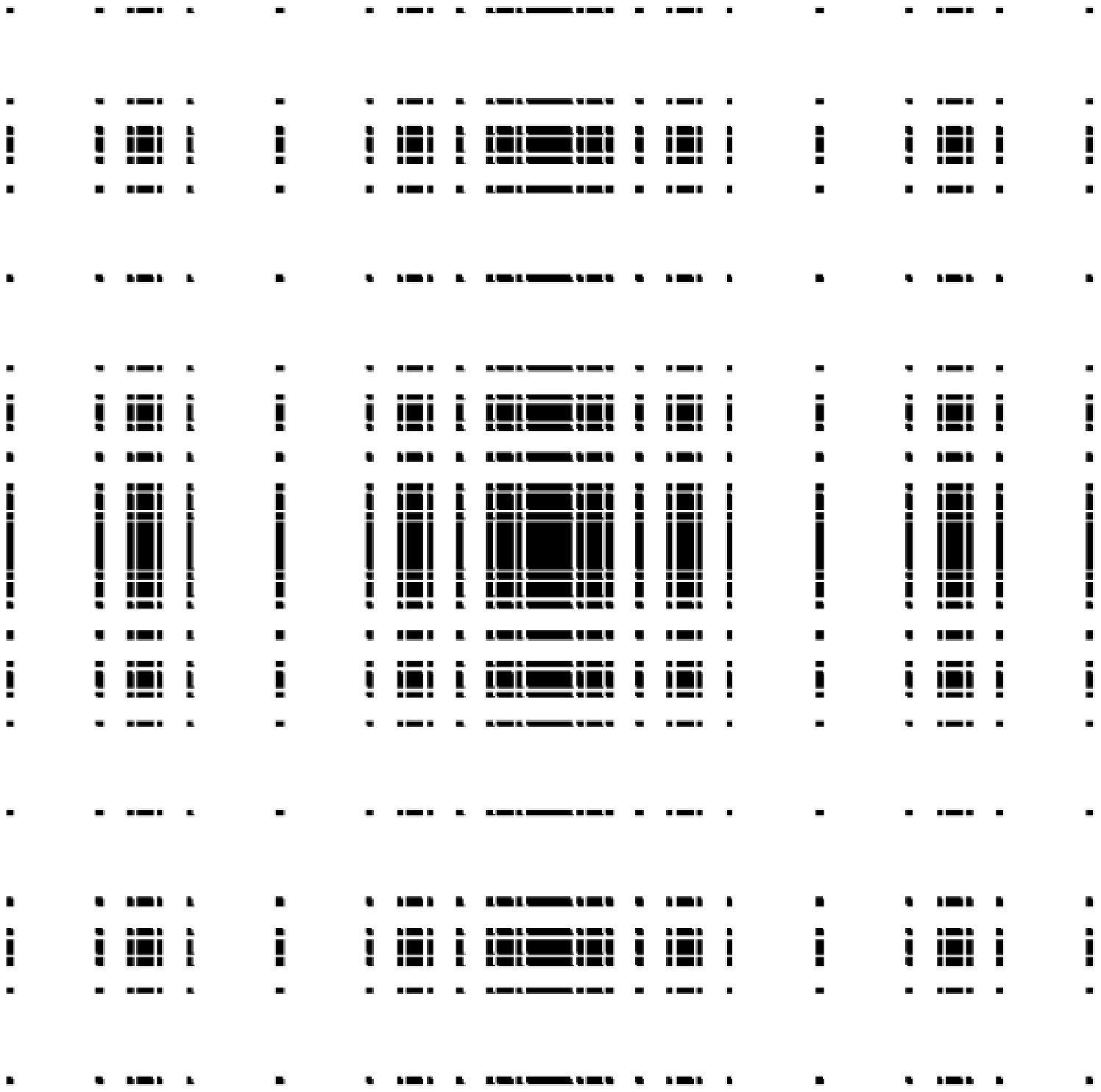}}
\label{f__K_figure}
\setlength{\unitlength}{1mm}
\begin{picture}(0,0)(+62,16)
\put(5,37){$\left\{ \vrule width 0mm height 25mm depth 0mm \right.$}
\put(3,38.5){\rnode{tail}}
\put(-33.5,48){\rnode{head}}
\put(-33.0,45){\tiny $Q$}
\put(-15,47){ $A_Q$}
\nccurve[linecolor=black,linewidth=0.5pt,angleA=60,angleB=170]{<-}{head}{tail}
\psline[linewidth=.1mm, linecolor=red](-48.6mm,17.5mm)(-48.6mm,31.1mm)
\psline[linewidth=.1mm, linecolor=red](-44.4mm,17.5mm)(-44.4mm,31.1mm)
\psline[linewidth=.1mm, linecolor=red](-48.6mm,31.1mm)(-44.4mm,31.1mm)
\psline[linewidth=.1mm, linecolor=red](-48.6mm,17.5mm)(-44.4mm,17.5mm)
\put(-43,23){\small   $Sk(I_1\times I_2)$}
\end{picture}
\caption{Each cube $Q\in\cK^2(n_1,k_1)$  will be replaced by the affine image of $\cK^2(n_2,k_2)$.
We also illustrate one of the many rectangle-skeletons in $\Gamma(n_1,k_1)$.}
\end{figure}

We define
\begin{align}
\Gamma= \Sk\bigl([0,1)^d \bigr) \ \cup\  \bigcup_{l=1}^\infty \ \ \bigcup_{Q\in\cK_{l-1}}\ \Gamma_Q(n_l,k_l)
\ \cup \ \bigcap_{l =1}^\infty\   K_{l}.\end{align}
Note that $\Gamma$ is connected.
The proof of Theorem~\ref{T:1} now reduces to the following two propositions,
which  show that $\cH^1(\Gamma)<\infty$ and $\nu(\Gamma)>0$ for a certain explicit choice of parameters.

\begin{proposition}[The length of $\Gamma$]\label{L:length of Gamma}
Let $\delta>0$ and $n_1\in\Z$ be parameters so that
\begin{equation}\label{E:epsilon-delta}
18 d[ \delta + \delta \log(\delta^{-1})] \leq \log(3)
\end{equation}
and $k_1:=3\delta n_1\geq 1$ is an integer.  If $\Gamma$ is the curve defined above with parameters
$n_l=ln_1$ and $k_l=lk_1$, then
\begin{align}\label{E:length Gamma}
\cH^1(\Gamma) &\leq 3d2^{d}e^{3dn_1[\delta+\delta\log(\delta^{-1})]}
\end{align}
\end{proposition}

\begin{proof}
By (2.7) we have
$$\cH^1\biggl(\bigcap_{l =1}^{\infty} K_l\biggr) \leq \prod_{l =1}^{\infty} \bigl(3^{-n_l} e^{dk_l(1 + \log(\delta^{-1})}\bigr) =0.$$
 Hence by \eqref{E:length Gamma n k} and~\eqref{E:size K_l},
\begin{align*}
\cH^1(\Gamma) &\leq d2^{d-1} + \sum_{l=1}^\infty
    \sum_{Q\in\cK_{l-1}} d 2^{d} 3^{-(n_1+\cdots+n_{l-1})} \exp\bigl\{ d k_l [ 1 + \log(\delta^{-1})]\} \\
&\leq d2^{d} \bigg[ 1 + \sum_{l=1}^\infty 3^{-(n_1+\cdots+n_{l-1})}
    \exp\bigl\{ (k_1+\cdots+k_l) d[ 1+ \log(\delta^{-1})] \bigr\} \biggr].
\end{align*}
Inserting the values of our parameters and performing a few elementary manipulations, we find
\begin{align*}
\cH^1(\Go) &\leq d2^{d-1} e^{3dn_1[\delta+\delta\log(\delta^{-1})]}
    \biggl[ 2 + \sum_{l=2}^\infty \exp\bigl\{ -\tfrac14 l(l-1)\log(3)n_1 \bigr\} \biggr]
\end{align*}
which yields \eqref{E:length Gamma} with a few more manipulations.
\end{proof}

\begin{proposition}[The measure of $\Gamma$]\label{L:measure of Gamma}
Let $\delta$ and $\{n_l,k_l\}_{l=1}^\infty$ be as in Proposition~{\rm\ref{L:length of Gamma}}.
Then
\begin{equation}
\nu(\Gamma) \geq \exp\biggl\{ - \frac{d e^{-2\delta^2 n_1}}{(1-e^{-2\delta^2 n_1})^2} \biggr\}.
\end{equation}
\end{proposition}

\begin{proof}
%First, we observe that
%$$
%\Gamma = \Go \ \cup \ \bigcap_{l>0} \; \overline{\cup\,\{Q:Q\in\cK_{l}\}}.
%$$
%Indeed, $\Go \,\cup\, \overline{\cup\,\{Q:Q\in\cK_{l}\}}$ is closed, contains $\Go$ for any $l>0$, and is contained within the
%$3^{-(n_1+\cdots+n_l)}$-neighborhood of $\Gamma$.
By the dominated convergence theorem,
\begin{align*}
\nu(\Gamma) \geq \lim_{l\to\infty} \nu\bigl(\overline{\cup\,\{Q:Q\in\cK_{l}\}}\bigr)
    \geq \lim_{l\to\infty} \nu\bigl(\cup\,\{Q:Q\in\cK_{l}\}\bigr).
\end{align*}
(In fact,  since doubling measures cannot charge straight lines, equality actually holds above,  but we will not need this.)
Since the cubes in $\cK_l$ have disjoint interiors, (2.5)
 and induction give us
\begin{align*}
  \nu\bigl(\bigcup\,\{Q:Q\in\cK_{l}\}\bigr) \geq \bigl[1 - e^{-2\delta^2 n_l} \bigr]^d 
\sum_{Q\in \cK_{l-1}} \nu(Q) \geq \prod_{j=1}^l \bigl[ 1 - e^{-2\delta^2 n_j} \bigr]^d.
\end{align*}
Inserting the values of our parameters and performing a few elementary
 manipulations, we conclude that 
% \begin{align*}
% \nu(\Gamma) &\geq \exp\biggl\{ d \sum_{j=1}^\infty \log(1-Z^j) \biggr\}
%     \geq \exp\biggl\{ 2d \sum_{1\leq k \leq j} Z^{jk} \biggr\}
%     \geq \exp\bigl\{ - \tfrac{2d Z} {(1-Z)^2}\bigr\}
% \end{align*}
\begin{align*}
\nu(\Gamma) &\geq \exp\biggl\{ d \sum_{j=1}^\infty \log(1-Z^j) \biggr\}
   \geq \exp\biggl\{ - d \sum_{j,k=1}^\infty Z^{jk} \biggr\}
   \geq \exp\bigl\{ - \tfrac{dZ} {(1-Z)^2}\bigr\}
\end{align*}
where $Z:=\exp\{-2\delta^2 n_1\}$.  That proves (2.14). 
\end{proof}

In closing, we note that the curve $\Gamma$ can be made to capture an arbitrarily large proportion of the $\nu$-mass of the unit cube; one merely chooses the parameter $n_1$ large (with $\delta$ fixed).

%%%%%%%%%%%%%%%%%%%%%%%%%%%%%%%%%%%%%%%%%%%%%%%%%%%%%%%%
%\section{Open questions}
%\begin{itemize}
%
%\item
%Suppose there is an $E$ with $H^1(E)<\infty$ and $\mu(E)>0$. Is it true that $\mu$ charges some rectifiable curve?
%What if you assume $H^1(E)=0$?  How about just for our $Kahane^d$ measure?
%
%\item
%suppose $\Gamma\subset Q_0$ is a rectifiable curve so that
%$\mu(\Gamma)>(1-\epsilon)\mu(Q_0)$.
%What is the best $C_\delta(\epsilon)$ taht we can find such that
%$ell(\Gamma)<C_\delta(\epsilon)$?  (again, start with our Kahane measure).
%Note: an upper bound is given by Gilad Lerman's thesis + a stopping time.
%we should be able to get the same with far less work...
%It would be interesting if we get a matching lower bound.%
%\end{itemize}

%%%%%%%%%%%%%%%%%%%%%%%%%%%%%%%%%%%%%%%%%%%%%%%%%%%%%%%%%%%%%%%%%%%%%%%%%%%%%%%%%%%%%%%%%%%%%%%%%%%%%%%%%%%%%%%%%%%%%%%%%%%%%%

\end{document}